\documentclass{article} 
\usepackage{iclr2026_conference,times}

\usepackage{amsmath,amsfonts,bm}

\def\eqref#1{equation~\ref{#1}}

\def\1{\bm{1}}

\DeclareMathAlphabet{\mathsfit}{\encodingdefault}{\sfdefault}{m}{sl}
\SetMathAlphabet{\mathsfit}{bold}{\encodingdefault}{\sfdefault}{bx}{n}

\usepackage{hyperref}
\usepackage{url}
\usepackage{graphicx}
\usepackage{subcaption}
\usepackage{booktabs}
\usepackage{amsthm}
\usepackage{algorithm}
\usepackage{algpseudocode}
\usepackage{multirow}
\theoremstyle{plain}
\newtheorem{lem}{Lemma}[subsection]

\title{Kinetic-based regularization: Learning spatial derivatives and PDE applications}

\author{%
  Abhisek Ganguly\thanks{Correspondence to abhisek@jncasr.ac.in} \, \& Santosh Ansumali \\
  Engineering Mechanics Unit \\
  Jawaharlal Nehru Centre for Advanced Scientific Research \\
  Jakkur, Bengaluru-560064, Karnataka, India \\
  \texttt{\{abhisek,ansumali\}@jncasr.ac.in} \\[2ex]
  \AND
  Sauro Succi \\
  Center for Life Nano-Neuroscience at La Sapienza \\
  Fondazione Istituto Italiano di Tecnologia \\
  Viale Regina Elena 291, 00161 Roma, Italy \\
  \texttt{sauro.succi@iit.it}
}
\iclrfinalcopy 
\begin{document}

\maketitle
\begin{abstract}
Accurate estimation of spatial derivatives from discrete and noisy data is central to scientific machine learning and numerical solutions of PDEs. We extend kinetic-based regularization (KBR), a localized multidimensional kernel regression method with a single trainable parameter, to learn spatial derivatives with provable second-order accuracy in 1D. Two derivative-learning schemes are proposed: an explicit scheme based on the closed-form prediction expressions, and an implicit scheme that solves a perturbed linear system at the points of interest. The fully localized formulation enables efficient, noise-adaptive derivative estimation without requiring global system solving or heuristic smoothing. Both approaches exhibit quadratic convergence, matching second-order finite difference for clean data, along with a possible high-dimensional formulation. Preliminary results show that coupling KBR with conservative solvers enables stable shock capture in 1D hyperbolic PDEs, acting as a step towards solving PDEs on irregular point clouds in higher dimensions while preserving conservation laws.
\end{abstract}

\section{Introduction} 
Over the years, regression has seen significant progress within the frame of machine learning. However, accurate derivative evaluation is a key area currently driving machine learning applications in computational physics (\cite{wang2020explicit}, \cite{abramavicius2025evaluationderivativesusingapproximate}).
Automatic differentiation (\cite{baydin2018automatic}) is now well-employed in modern deep learning frameworks, enabling the computation of spatial and temporal derivatives in physics-informed models and also forming the basis of backpropagation (\cite{rumelhart1986learning}). However, it also faces significant challenges (\cite{gholami2019anode},~\cite{margossian2019review},~\cite{huckelheim2024adpitfalls}). 

There has been a marked interest in deep learning approaches for solving PDEs, e.g. PINNs (\cite{raissi2019physics}). Recent efforts recognize the importance of conservative variants (\cite{Jagtap2020cPINNs}) that improve flux consistency. Though being parameter-heavy with expensive optimization, PINNs' functionality in higher dimensions cannot be ignored (\cite{kumar2026robust}). Related operator and derivative learning frameworks (\cite{lu2021learning},~\cite{li2021fourier},~\cite{Ledesma2020DifferentialNeuralNetworks}) further highlight the growing need for robust, physics-aware differentiation. Many such learning schemes, however, lack rigorous convergence and error analysis \cite{wang2022and}, which hinders systematic performance assessment.

Recently, kinetic-based regularization (KBR,~\cite{Ganguly_2025}) was introduced to revive Radial Basis Function networks, one of the earliest kernel-based neural models (\cite{broomhead1988rbf}). 
The physics-inspired kernel regression scheme adopts a fully localized formulation that avoids solving large-scale global systems of equations. This enables learnable noise removal using a single trainable parameter, independent of the problem dimensionality.

In this paper, we extend KBR to learn spatial derivatives with fully interpretable accuracy, avoiding direct kernel differentiation approach used by methods such as Gaussian Processes (\cite{rasmussen2006gaussian}) that is known to cause instabilities (\cite{gingold1977sph},~\cite{bardenhagen2001mpm},~\cite{monaghan2005sph}) in numerical solvers. We present explicit and implicit schemes of derivative extraction from learned fields, each with its own key strengths. We show that KBR is capable of stable integration into conservative 1D hyperbolic PDE solvers where traditional PINN-based learning approaches often saturate. This bridges kernel learning with traditional solvers, a potential step towards simulating PDEs on irregular high-dimensional point clouds while respecting conservation laws.

\section{Methodology}

The original KBR aims to fit an unknown function with a locally quadratic polynomial, but gives us no information about the fitting constants involved. The fit is ensured by a point-wise enforcement of lower order moments, arising from the disagreement between continuous and discrete statistics. The desired local quadratic fit using KBR for some $\pmb{x}\text{ for } \pmb{x} \subset \mathbb{R}^D$ is:
\begin{gather}\label{eq:general}
    \phi(\pmb{x}) = a + \pmb{b} \cdot \pmb{x} + \pmb{C}:\pmb{x}\pmb{x}^{\dagger},
\end{gather}
However, the original scheme fails to impose strict second-order accuracy anywhere except the training points. We address this issue by introducing an improved correction that eliminates this error entirely from the previously unseen points as well (see $\S$~\ref{asec:secendordercorrection} for details). 
The spatial derivatives of the fitted function can then be given by,
\begin{gather}\label{eq:1stder}
    \frac{\partial \phi}{\partial x_\alpha  } = b_\alpha \delta_{\alpha\gamma} + c_{\gamma\kappa}x_\gamma\delta_{\alpha\kappa} + c_{\gamma\kappa}x_\kappa\delta_{\alpha\gamma}, \qquad     \frac{\partial^2 \phi}{\partial x_\alpha \partial x_\beta } = c_{\gamma\kappa}\delta_{\beta\gamma}\delta_{\alpha\kappa} + c_{\gamma\kappa}\delta_{\beta\kappa}\delta_{\alpha\gamma}.  
\end{gather}
where \( \phi'(x) \) denotes the gradient (first derivative), and \( \phi''(x) \) denotes the Laplacian (second derivative). In 1D, the equations simplify to $\frac{\partial \phi}{\partial x} = b+2cx$ and $\frac{\partial^2 \phi}{\partial x^2} = 2c$. The information about the spatial derivatives are present in the fit constants. We present two ways to extract this information: 
\begin{itemize}
    \item The \textbf{explicit} way is to calculate each of these quantities sequentially as closed-form expressions at a test point $x$ (refer to $\S$~\ref{asec:explicit} for details). 
    \item Alternatively in the \textbf{implicit} way, we perturb the test point \(x\) by a small \(\varepsilon\), evaluate the predictions at \(x_0 \pm \varepsilon\), and recover the required local constants implicitly by solving the resulting system of equations (see \(\S\)~\ref{asec:implicit}). This scheme can be expanded to higher dimensions by design.
\end{itemize}

Derivative extraction follows KBR training (with exact second-order correction), with no additional information requirement. We demonstrate the performance of two schemes in 1D using a number of tests below, with the experimentation details given in $\S$~\ref{asec:num_exp}. 

\section{Results}
\textbf{Explicit scheme is more stable for unknown data:} We test the derivative learning schemes on an unknown test data of $5,000$ random points for two 1D functions, namely the Camel function with strong multiplicative coupling and the fully separable Rastrigin function. The general functions are given as follows:
{\small
\begin{equation}\label{eq:camel}
  f_{\text{camel}}(\bm{x}) = \frac{1}{2(k \sqrt{\pi})^{D}} \left( 
  \exp\left(-\sum_{i=1}^D\frac{\left(x_i - \frac{1}{3}\right)^2}{k^2}\right) 
  + \exp\left(-\sum_{i=1}^D\frac{\left(x_i - \frac{2}{3}\right)^2}{k^2}\right) 
  \right)
\end{equation}

\begin{equation}
  f_{\text{ras}}(\bm{x}) = \sum_{i=1}^{D} 
  \left[ x_i^2 - A \cos(2\pi x_i) + A \right]
\end{equation}
}
where \( \bm{x} = (x_1, x_2,......,x_D) \in \mathbb{R}^D \), with $k = 0.2$ and $A = 10$. For the present study, $D=1$. We refer to Fig.~\ref{fig:1d_implicit_vs_explicit}. It is observed that the explicit scheme shows a more stable behavior towards convergence, as compared to the implicit algorithm. For the Camel function, we observe a threshold sampling size beyond which the explicit scheme's performance sharply starts approaching second-order finite difference accuracy. This tendency is less apparent for the mathematically simpler Rastrigin function.
\begin{figure}[h]
    \centering
    \includegraphics[width=0.8\linewidth]{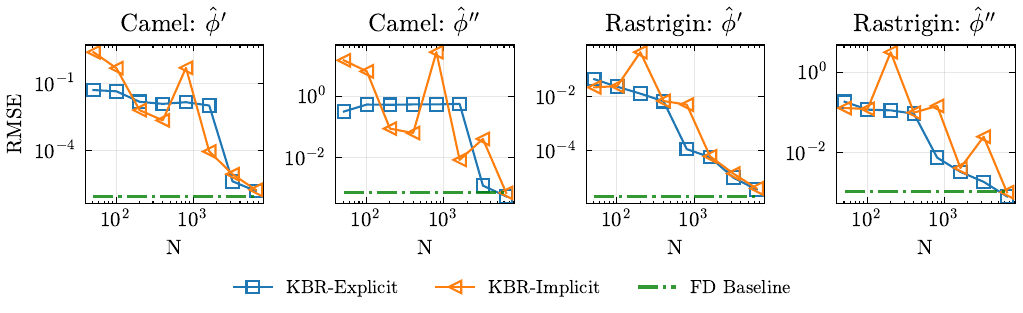}
    \caption{We show convergence of implicit and explicit derivative prediction schemes ($  \hat{\phi}'  $ and $  \hat{\phi}''  $) for two 1D functions with random sampling. RMSE is computed on a random test set of $5,000$ points, benchmarked against second-order non-uniform FD.}
    \label{fig:1d_implicit_vs_explicit}
\end{figure}

\textbf{Benchmarking on clean data:} We then move on to fields with known function values, eliminating the step of learning the function on unseen points. As a simple baseline validation study, we test our scheme against the so called Differential Neural Networks (\cite{Ledesma2020DifferentialNeuralNetworks}), or DNN, from which the exact test conditions can be referred. Tab.~\ref{tab:DNNvsKBR} shows the present KBR scheme outperforming DNN significantly. The second-order accuracy is evident in the $  x^2  $ case, where explicit-scheme derivative errors reach machine precision.

\begin{table}[h]
    \centering
    \footnotesize
    \caption{Comparison of non-normalized MSE (DNN vs KBR) on deployment/test data}
    \begin{tabular}{@{} l c c c c c @{}}
        \toprule
        $\phi(x)$ & DNN & \multicolumn{2}{c}{KBR-Implicit} & \multicolumn{2}{c}{KBR-Explicit} \\
        \cmidrule(lr){3-4} \cmidrule(lr){5-6}
                  & $\phi'$ & $\phi'$ & $\phi''$ & $\phi'$ & $\phi''$ \\
        \midrule
        $\sin(x)$  & $7.5 \times 10^{-6}$ & $6.60 \times 10^{-12}$ & $2.22 \times 10^{-4}$ & $3.12 \times 10^{-12}$ & $6.50 \times 10^{-5}$ \\
        $x^2$      & $1.1 \times 10^{-3}$ & $8.14 \times 10^{-13}$ & $3.27 \times 10^{-4}$ & $3.24 \times 10^{-19}$ & $3.07 \times 10^{-14}$ \\
        $\ln(x)$   & $3.6 \times 10^{-6}$ & $8.92 \times 10^{-12}$ & $2.52 \times 10^{-4}$ & $3.88 \times 10^{-12}$ & $1.34 \times 10^{-4}$ \\
        \bottomrule
    \end{tabular}
    \label{tab:DNNvsKBR}
\end{table}

Convergence is tested on the 1D Camel function with increasing random sample sizes $  N  $, comparing derivative errors against non-uniform FD (\cite{Fornberg1988,LeVeque2007}). The top panel of Fig.~\ref{fig:kbr_fdm_and_noisy} shows stable convergence for our schemes. While FD accuracy degrades on highly irregular grids (especially for $  \phi''  $, \cite{crowder1971errors,dehoog1985rate}), our method performs comparably. Both the theoretically second-order finite-difference (FD) scheme and the present scheme exhibit second-order convergence for $\phi'$ and first-order convergence for $\phi''$.

This scheme is generalizable in higher dimensions. A qualitative implementation result for the 2D Camel function is shown in $\S$~\ref{asec:implicit}.

\textbf{Implicit scheme outperforms explicit scheme for noisy data:} We refer to the bottom panel of Fig.~\ref{fig:kbr_fdm_and_noisy}. The implicit scheme is more robust with a much lesser error growth as the degree of corruption increases in the training data. For the smoothing cubic spline, a pre-determined window length proportional to the noise variance is used as a reference (see $\S$~\ref{asec:num_exp}), whereas the KBR schemes aren't given any such pre-conditioning. It is apparent that the standard FD requires regularization for noisy data and isn't applicable here as is (\cite{chartrand2011numerical}).

\begin{figure}[h]
    \centering
    \small
    \includegraphics[width=0.7\linewidth]{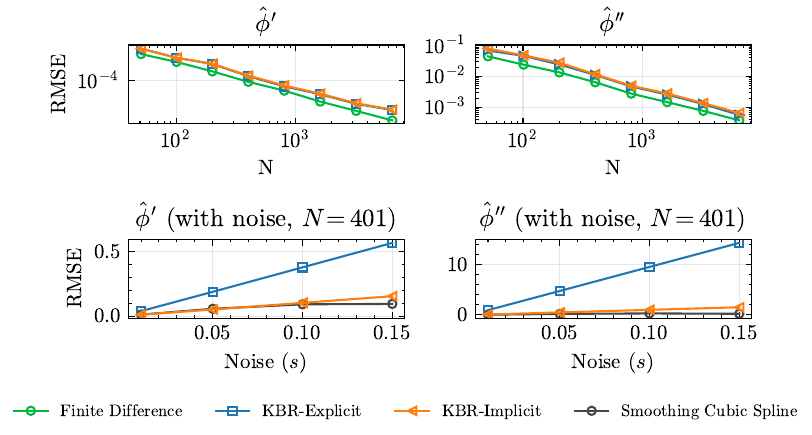}
    \caption{Convergence plot (top panel) comparing errors from KBR-predicted spatial derivatives with second-order FD for the 1D Camel function on a non-uniform grid. Bottom panel shows the errors in extracted derivative errors when the training data is corrupted with varying degrees of gaussian noise. Cubic spline-smoothened results are given as a reference.}
    \label{fig:kbr_fdm_and_noisy}
\end{figure}
\textbf{Application of KBR in PDE Solutions:}
The general conservative form of the hyperbolic PDEs considered in this work is given by
\begin{equation}
\footnotesize
\frac{\partial \phi}{\partial t} + \frac{\partial F(\phi)}{\partial x} = 0,
\end{equation}
where $\phi$ denotes the solution field and $F(\phi)$ represents the corresponding flux. We integrate KBR into conservative hyperbolic PDE solvers by replacing conventional flux evaluations with KBR-based predictions, thereby promoting consistency with fundamental conservation laws beyond heuristic enforcement. The single trainable parameter $\theta$ may be retrained at every time step or periodically over multiple steps. Training is performed using grid-point data, while cell interfaces are treated as deployment locations for flux prediction owing to minimal training cost. The proposed approach is evaluated on two benchmark problems to assess the integration stability: the one-dimensional inviscid Burgers’ equation with a Riemann-type initial condition on a uniform grid, and the one-dimensional Euler equations for the Sod shock tube problem (\cite{sod1978survey}) on a non-uniform grid. Implementation details are provided in $\S$~\ref{asec:num_exp}. 

\begin{figure}[h]
    \centering
    \includegraphics[width=0.7\linewidth]{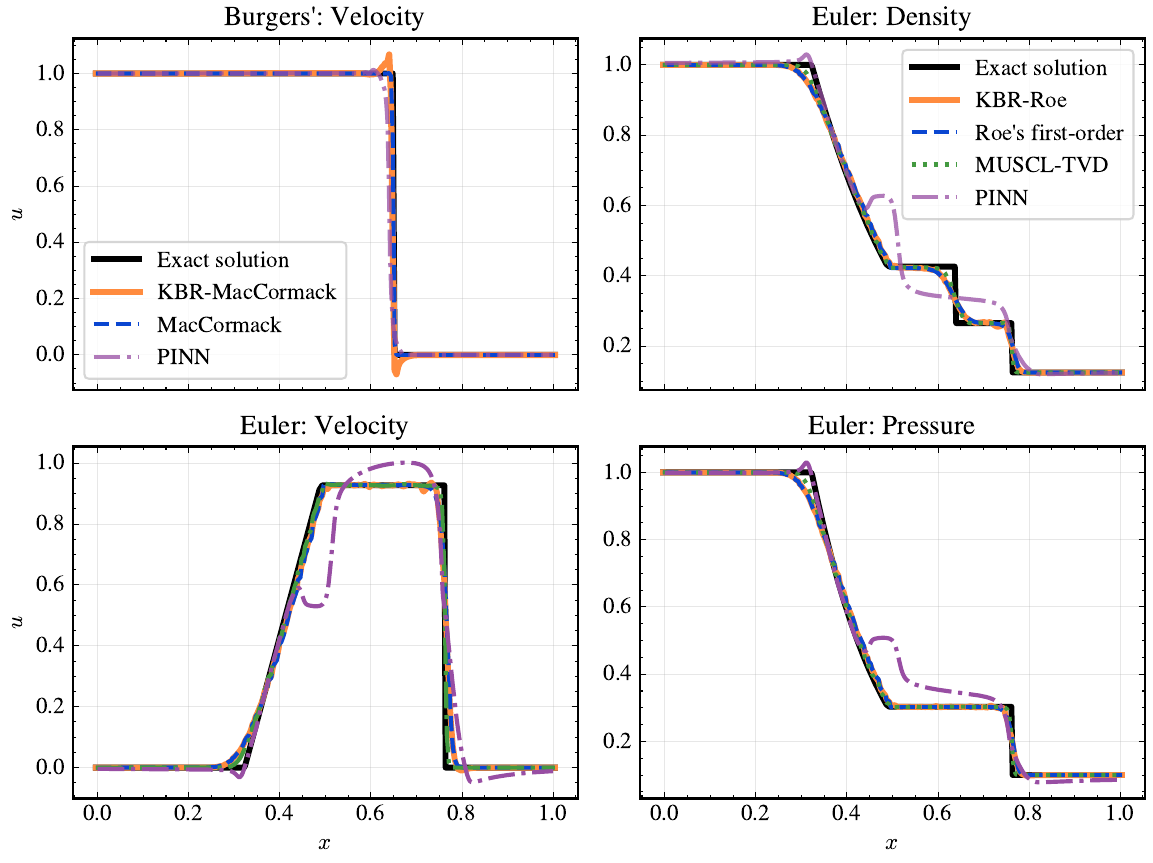}
    \caption{The figure shows the numerical solution at time $t=0.3$, and $t = 0.15$ for 1D inviscid Burgers' and 1D Euler equations respectively, using strict FD and KBR-integrated schemes on a grid size of 251 points. Trained PINN predictions and exact solutions are given as a reference.}
    \label{fig:PDEs}
\end{figure}

\begin{table}[h]
\caption{Shock-capturing metrics for density $\rho$ in the Sod shock tube problem on a non-uniform grid (KBR-Roe, Roe first-order, MUSCL-TVD). 
Region~1: shock $[0.62, 0.67]$, post-shock $[0.64, 0.72]$; 
Region~2: shock $[0.73, 0.77]$, post-shock $[0.755, 0.80]$.}
\centering
\tiny
\renewcommand{\arraystretch}{1.3}
\begin{tabular}{llrrrr}
\toprule
Region &Metrics & KBR-Roe & Roe & MUSCL-TVD \\
\midrule
\multirow{6}{*}{Shock 1 (Region 1)} 
& L1 Error (Global)          & $9.409\times10^{-3}$ & $1.014\times10^{-2}$ & $5.311\times10^{-3}$ \\
& L$\infty$ Error (Global)   & $9.305\times10^{-2}$ & $9.021\times10^{-2}$ & $8.614\times10^{-2}$ \\
& Shock Thickness   & $2.874\times10^{-2}$ & $3.221\times10^{-2}$ & $2.450\times10^{-2}$ \\
& Post-shock Osc.   & $4.919\times10^{-2}$ & $5.162\times10^{-2}$ & $4.392\times10^{-2}$ \\
& Total Variation   & $1.010\times10^{-1}$ & $9.646\times10^{-2}$ & $1.265\times10^{-1}$ \\
\midrule
\multirow{6}{*}{Shock 2 (Region 2)} 
& L1 Error          & $9.409\times10^{-3}$ & $1.014\times10^{-2}$ & $5.311\times10^{-3}$ \\
& L$\infty$ Error   & $9.305\times10^{-2}$ & $9.021\times10^{-2}$ & $8.614\times10^{-2}$ \\
& Shock Thickness   & $3.087\times10^{-2}$ & $1.638\times10^{-2}$ & $7.641\times10^{-3}$ \\
& Post-shock Osc.   & $6.858\times10^{-2}$ & $7.641\times10^{-2}$ & $5.433\times10^{-2}$ \\
& Total Variation   & $1.067\times10^{-1}$ & $7.656\times10^{-2}$ & $1.295\times10^{-1}$ \\
\bottomrule
\end{tabular}
\label{tab:sod_density_metrics}
\end{table}

The KBR schemes stay stable across both PDEs (Fig.~\ref{fig:PDEs}), though standard Gibbs oscillations (\cite{Wilbraham1848}) appear for the Burgers' equation. This is because unlike the standard MacCormack, the KBR-integrated scheme uses a predicted second-order flux value at the cell faces. For the Euler equations, the KBR-integrated first-order Roe scheme is shown to achieve performance comparable to the standard Roe scheme in terms of shock resolution and error metrics (Table~\ref{tab:sod_density_metrics}), with no error blow-ups over time. The results are also compared with those obtained using the higher-order flux-limited MUSCL–TVD scheme (\cite{laney1998computational}), though a superiority is not expected given the order of MUSCL–TVD. Note that the variation in the observation window does not change relative trends significantly. The incorporation of KBR within TVD (Total Variation Diminishing) and WENO (Weighted Essentially Non-Oscillatory) schemes is currently under investigation.

In the examples considered, training of KBR is significantly faster than PINNs, using lesser computing resources. The objective of this demonstration is not to establish superiority over standard numerical schemes, but to illustrate the dynamic stability and competitive behavior of the proposed learning-incorporated approach. Integrating learning-based components with traditional numerical PDE solvers in this manner may provide an additional framework for robust spatial derivative evaluation, particularly for high-dimensional problems and unstructured grids.

\section{Conclusion and Outlook}
We have presented two spatial derivative learning schemes as an expansion of KBR, applicable in known, unknown and corrupt fields primarily in 1D and preliminary in 2D. Termed here as implicit and explicit, both the schemes achieve second-order FD accuracy for derivative learning while remaining reasonably robust to noise. Quantitative evaluation using the underlying theory and convergence rates provides a principled measure of performance, thereby enhancing the interpretability of the learning scheme.
The presented KBR-integrated PDE simulations act as a step towards incorporating machine learning into PDE solvers in a manner that adheres to the basic conservation laws in a non-heuristic way.
Future work targets improved implicit scheme stability and extension to high-dimensional PDE simulations in random point-clouds. The relevant scripts along with the result datasets are available at 
\url{https://github.com/abhishekganguly808/kbr-derivatives}.


\subsubsection*{Acknowledgments}
A. G. thanks Guruprasad S for all the meaningful discussions on the numerical solutions of hyperbolic PDEs.
\newpage

\bibliography{no_doi}
\bibliographystyle{iclr2026_conference}

\newpage
\appendix
\section{Appendix}
\subsection{Exact Second-Order Correction}\label{asec:secendordercorrection}

The original KBR self-correction yields interpolated second-moment errors on test points, and the second-order accuracy on unseen points is no exact in the original work. We prove exact second-order accuracy via targeted error removal, and address this shortcoming.

\begin{figure}[h]
    \centering
    \includegraphics[width=0.7\linewidth]{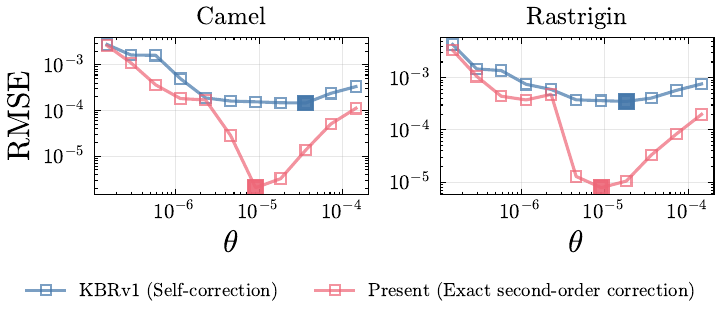}
    \caption{Comparison between the loss function landscapes for the case of self-correction and the proposed exact second-order correction, with respect to the single parameter $\theta$ and validation error in KBR. Training data consists of $501$ randomly sampled points, with an 80:20 split between train and validation.}
    \label{afig:energy_minima_camel}
\end{figure}

\begin{lem}[Exact Second-Order KBR]
Let $\{x_i,\phi(x_i)\}_{i=1}^N$ be a set of training points, For fully converged Lagrange multipliers $\tilde{x}$, define $\tilde{\xi}_i = x_i - \tilde{x}$. The corrected second-order prediction,
\begin{equation}\label{eq:new_2nd_corr}
\hat{\phi}^{(2)}(x) = \sum_i \phi(x_i) P(\|\tilde{\xi}_i\|,\theta) 
- \Theta(x) \sum_i \frac{\phi(x_i)-\phi^{(1)}(x)}{\Theta(x_i)} P(\|\tilde{\xi}_i\|,\theta),
\end{equation}
satisfies $\phi^{(2)}(x) = \phi(x)$ exactly for quadratic $\phi(x) = a+bx+cx^2$ at test points $x\notin\{x_i\}$.
\end{lem}

\begin{proof}
Define second-moment errors:\\
$\Theta(x) = \sum_i x_i^2 P(\|\tilde{\xi}_i\|,\theta) \, - x^2$ and $\Theta(x_k) =  \sum_i x_i^2 P(\|\tilde{\xi}_i\|,\theta) - x_k^2$, for $k\in i$. 

Converged first moments give exact first order prediction, $\hat{x}_i = x_i$. The original correction therefore expands as:
\begin{align}
\phi^{(2)}(x) &= \sum_i \bigl(2\phi(x_i) - \phi^{(1)}(x_i)\bigr) P(\|\tilde{\xi}_i\|,\theta) \notag \\
&= a + bx + c\sum_i (2x_i^2 - \widehat{x^2}_i) P(\|\tilde{\xi}_i\|,\theta) \notag\\
&= \phi(x) + c\,\bigl(\Theta(x) - \hat{\Theta}(x)\bigr)
\end{align}
where $\hat{\Theta}(x) = \sum_i \Theta(x_i) P(\|\tilde{\xi}_i\|,\theta)$ are the interpolated training errors from $\Theta(x_i)$, which fails to cancel out the thermal error at $x$. Also, $\widehat{x^2_i} = \sum_i x_i^2 P(\|\tilde{\xi}_i\|,\theta) $

For Eq.~\ref{eq:new_2nd_corr}, the correction term simplifies:
\begin{align}
&\Theta(x) \sum_i \frac{\phi(x_i)-\phi^{(1)}(x)}{\Theta(x_i)} P(\|\tilde{\xi}_i\|,\theta) \notag\\
&= c\,\Theta(x) \sum_i \frac{x_i^2 - (x_i^2 + \Theta(x_i))}{\Theta(x_i)} P(\|\tilde{\xi}_i\|,\theta) = c\,\Theta(x)
\end{align}
Thus, $\phi^{(2)}(x) = \sum_i \phi(x_i) P(\|\tilde{\xi}_i\|,\theta) - c \, \Theta(x) = \phi(x)$.
\end{proof}

This ensures exact second-order representation on unseen points locally. The immediate effect can be seen in Fig.~\ref{afig:energy_minima_camel}, where the current corrections lead to a significantly lower validation error with respect to the trainable kernel parameter $\theta$.

\begin{algorithm}[H]
\caption{Exact second-order correction KBR}
\label{alg:kbr_explicit}
\begin{algorithmic}[1]

\Require Training data $\{x_i,\phi(x_i)\}_{i=1}^N$, test points $\{x\}$, kernel parameter $\theta$, max iterations $K$

\Ensure Second-order prediction $\hat{\phi}(x)$

\For{each test point $x$}

    \State $\tilde{x} \leftarrow$ LagrangeMultiplier$(x,\{x_i\},\theta,K)$  \hfill $\triangleright$ Enforce moment constraint

    \For{$i=1$ to $N$}
        \State $P_i \leftarrow P(\|x_i-\tilde{x}\|,\theta)$
    \EndFor

    \State $\phi^{(1)}(x) \leftarrow \sum_i \phi(x_i)P_i$  \hfill $\triangleright$ First-order prediction

    \State $\Theta(x) \leftarrow \sum_i (x_i^2-x^2)P_i$  \hfill $\triangleright$ Second moment error

    \For{$i=1$ to $N$}
        \State $\Theta(x_i) \leftarrow \sum_j (x_j^2-x_i^2)P_j$
    \EndFor

    \For{$i=1$ to $N$}
        \State $c_i \leftarrow \dfrac{\phi(x_i)-\phi^{(1)}(x)}{\Theta(x_i)}$
    \EndFor

    \State $\hat{\phi}(x) \leftarrow \phi^{(1)}(x)+\Theta(x)\sum_i c_i P_i$  \hfill $\triangleright$ Second-order correction

\EndFor

\end{algorithmic}
\end{algorithm}

\subsection{Explicit scheme for derivative calculation}
\label{asec:explicit}

\textbf{Laplacian Calculation:}
We assume that the local field in the neighborhood of a test point $x_0$
is well represented by the quadratic polynomial $\phi(x) = a + bx + cx^2$, the exact second derivative is then $\phi''(x_0)=2c$. Consider the uncorrected prediction at $x_0$,
\begin{align}
\phi^{(0)}(x_0)
&= \sum_i \phi(x_i) P(\|\xi_i\|,\theta)
= a + b \hat{x}_0 + c \sum_i x_i^2 P(\|\xi_i\|,\theta),
\end{align}
where $\hat{x}_0=\sum_i x_i P(\|\xi_i\|,\theta)$.
Rewriting,
\begin{equation}
\phi^{(0)}(x_0)
=
\phi(x_0)
+
b(\hat{x}_0-x_0)
+
c\,\Theta^{(0)}(x_0),
\end{equation}

After enforcing first-moment consistency via Lagrange multipliers,
$\hat{x}_0=x_0$, and the first-order corrected prediction becomes
\begin{equation}
\phi^{(1)}(x_0)
=
\sum_i \phi(x_i) P(\|\tilde{\xi}_i\|,\theta)
=
\phi(x_0)+c\,\Theta(x_0),
\end{equation}
where $\Theta$ is defined in the proof of the previous section  .

Subtracting the two predictions isolates the quadratic contribution,
\begin{equation}
\phi^{(0)}(x_0)-\phi^{(1)}(x_0)=c\,\Theta(x_0).
\end{equation}
Since $\phi''(x_0)=2c$, the Laplacian is obtained as
\begin{equation}
\phi''(x_0)
=
\frac{2}{\Theta(x_0)}
\bigl[\phi^{(0)}(x_0)-\phi^{(1)}(x_0)\bigr].
\end{equation}
If the field value $\phi(x_0)$ is already known, it may be used directly in
place of $\phi^{(1)}(x_0)$.

\medskip
\textbf{Gradient Calculation:}
Using the same quadratic representation,
$\phi(x)=a+bx+cx^2$, the exact gradient at $x_0$ is
\begin{equation}
\phi'(x_0)=b+2cx_0.
\end{equation}

The uncorrected prediction can again be written as
\begin{equation}
\phi^{(0)}(x_0)
=
\phi(x_0)
+
b(\hat{x}_0-x_0)
+
c\,\Theta^{(0)}(x_0),
\end{equation}
while the first-moment corrected prediction satisfies
\begin{equation}
\phi^{(1)}(x_0)=\phi(x_0)+c\,\Theta(x_0).
\end{equation}

Subtracting the two expressions yields
\begin{equation}
\phi^{(0)}(x_0)-\phi^{(1)}(x_0)
=
b(\hat{x}_0-x_0)
+
c\bigl(\Theta^{(0)}(x_0)-\Theta(x_0)\bigr).
\end{equation}
The gradient is recovered as
\begin{equation}
\phi'(x_0)
=
\frac{1}
{\hat{x}_0-x_0} \left(\phi^{(0)}(x_0)-\phi^{(1)}(x_0)
-
\frac{\phi''(x_0)}{2}
\bigl[\Theta^{(0)}(x_0)-\Theta(x_0)\bigr]\right).
\end{equation}

\begin{algorithm}[H]
\caption{Explicit scheme for KBR derivatives}
\label{alg:kbr_explicit_derivative}
\begin{algorithmic}[1]

\Require Training data $\{x_i,\phi(x_i)\}_{i=1}^N$, test points $\{x\}$, kernel parameter $\theta$, max iterations $K$

\Ensure Gradient $\phi'(x)$ and Laplacian $\phi''(x)$

\For{each test point $x$}

    \State $\tilde{x} \leftarrow$ LagrangeMultiplier$(x,\{x_i\},\theta,K)$
    \hfill $\triangleright$ Enforce first moment

    \For{$i=1$ to $N$}
        \State $P_i \leftarrow P(\|x_i-\tilde{x}\|,\theta)$
        \State $P_i^{(0)} \leftarrow P(\|x_i-x\|,\theta)$
    \EndFor

    \State $\phi^{(0)}(x) \leftarrow \sum_i \phi(x_i)P_i^{(0)}$
    \hfill $\triangleright$ Uncorrected prediction

    \State $\phi^{(1)}(x) \leftarrow \sum_i \phi(x_i)P_i$
    \hfill $\triangleright$ Corrected prediction

    \State $\Theta^{(0)}(x) \leftarrow \sum_i (x_i^2-x^2)P_i^{(0)}$

    \State $\Theta(x) \leftarrow \sum_i (x_i^2-x^2)P_i$

    \State $\phi''(x) \leftarrow \dfrac{2\bigl(\phi^{(0)}(x)-\phi^{(1)}(x)\bigr)}{\Theta(x)}$
    \hfill $\triangleright$ Laplacian

    \State $\hat{x} \leftarrow \sum_i x_i P_i^{(0)}$

    \State $\phi'(x) \leftarrow
    \dfrac{1}{\hat{x}-x}
    \left(
    \phi^{(0)}(x)-\phi^{(1)}(x)
    -
    \dfrac{\phi''(x)}{2}
    \bigl[\Theta^{(0)}(x)-\Theta(x)\bigr]
    \right)$
    \hfill $\triangleright$ Gradient

\EndFor

\end{algorithmic}
\end{algorithm}

\subsection{Implicit scheme for derivative calculation}\label{asec:implicit}

Consider the quadratic fit is given as
\begin{equation}
\phi(x) = a + bx +cx^2
\end{equation}
For a point $x_0$ of interest, three predictions can be made using slight perturbations in the Lagrange multipliers $\tilde{\xi}_i$, which goes into the kernel function $P$. It can be represented as:
\begin{gather}
    \tilde{\xi}_i = x - \tilde{x}_0 , \notag \\
    \tilde{\xi}^-_i = x - (\tilde{x}_0 - \varepsilon ), \notag \\
    \tilde{\xi}^+_i = x - (\tilde{x}_0 +\varepsilon ),
\end{gather}
where, $\varepsilon$ can be a small percentage of effective distance between the points, to preserve the local nature of derivative calculation. The problem can then be represented as:
\[
\underbrace{
\begin{bmatrix}
\sum_i P(\tilde{\xi}_i) & \sum_i x_i P({\tilde{\xi}}_i) & \sum_i x^2_i P({\tilde{\xi}}_i) \\
\sum_i P(\tilde{\xi}^-_i) & \sum_i x_i P(\tilde{\xi}^-_i) & \sum_i x_i^2 P(\tilde{\xi}^-_i)\\
\sum_i P(\tilde{\xi}^+_i) & \sum_i x_i P(\tilde{\xi}^+_i) & \sum_i x_i^2 P(\tilde{\xi}^+_i) 
\end{bmatrix}}_{A}
\underbrace{
\begin{bmatrix}
a \\ b \\ c
\end{bmatrix}}_{X}
=
\underbrace{
\begin{bmatrix}
\hat{\phi}(x_0) \\ \hat{\phi}^+(x_0) \\ \hat{\phi}^-(x_0)
\end{bmatrix}}_{B}
\]
In the present study, $\varepsilon$ is chosen as $0.025\sqrt{\theta_o}$, where $\theta_o$ denotes the optimal kernel parameter obtained from the relation $\theta_o/d_{\text{typ}}^2 = k$. The dimensionless parameter $k$ is optimized instead of $\theta$. Here, $d_{\text{typ}}$ represents the effective local distance, evaluated as the average $k$-nearest-neighbor distance using five points. Empirically, a sweep of $10$–$15$ iterations over $k \in [0.01,100]$ is found to yield favorable local minima of the loss function. 

\begin{algorithm}[H]
\caption{Implicit scheme for KBR derivatives}
\label{alg:kbr_implicit_matrix}
\begin{algorithmic}[1]

\Require Training data $\{x_i,\phi(x_i)\}_{i=1}^N$, test points $\{x\}$, kernel parameter $\theta$, max iterations $K$

\Ensure Gradient $\phi'(x)$ and Laplacian $\phi''(x)$

\For{each test point $x$}

    \State $\tilde{x} \leftarrow$ LagrangeMultiplier$(x,\{x_i\},\theta,K)$
    \hfill $\triangleright$ Moment consistency

    \State $\tilde{x}^- \leftarrow \tilde{x}-\varepsilon$, \quad $\tilde{x}^+ \leftarrow \tilde{x}+\varepsilon$
    \hfill $\triangleright$ Perturbation

    \For{$i=1$ to $N$}
        \State $P_i \leftarrow P(\|x_i-\tilde{x}\|,\theta)$
        \State $P_i^- \leftarrow P(\|x_i-\tilde{x}^-\|,\theta)$
        \State $P_i^+ \leftarrow P(\|x_i-\tilde{x}^+\|,\theta)$
    \EndFor

    \State Assemble matrix $A$:
    \[
    A=
    \begin{bmatrix}
    \sum_i P_i & \sum_i x_i P_i & \sum_i x_i^2 P_i \\
    \sum_i P_i^- & \sum_i x_i P_i^- & \sum_i x_i^2 P_i^- \\
    \sum_i P_i^+ & \sum_i x_i P_i^+ & \sum_i x_i^2 P_i^+
    \end{bmatrix}
    \]

    \State Assemble vector $B$:
    \[
    B=
    \begin{bmatrix}
    \hat{\phi}(x)\\
    \hat{\phi}^-(x)\\
    \hat{\phi}^+(x)
    \end{bmatrix}
    \]

    \State Solve $AX=B$ for $X=[a\;b\;c]^T$
    \hfill $\triangleright$ Local quadratic fit

    \State $\phi''(x)\leftarrow 2c$
    \hfill $\triangleright$ Laplacian

    \State $\phi'(x)\leftarrow b + 2cx$
    \hfill $\triangleright$ Gradient

\EndFor

\end{algorithmic}
\end{algorithm}

\begin{figure}[h!]
    \centering
    \includegraphics[width=\linewidth]{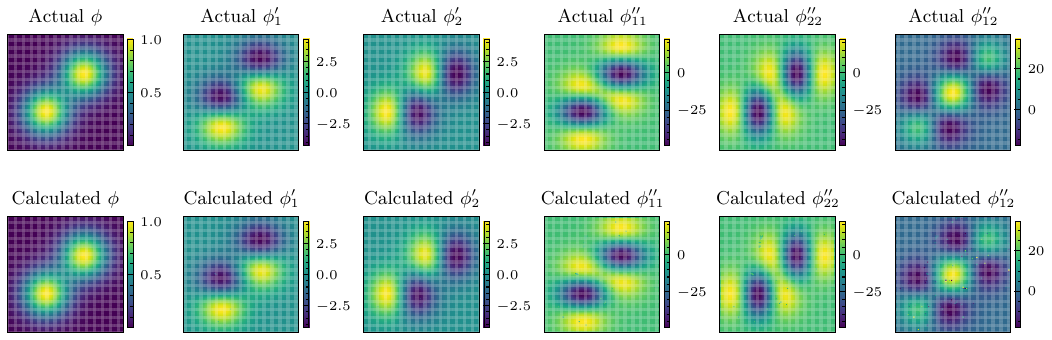}
    \caption{The actual and predicted fields, gradients, Laplacian and Hessian is shown for the 2D Camel function on a uniform $101\times 101$ grid using KBR Implicit spatial derivative scheme.}
    \label{afig:camel_2d}
\end{figure}

\textbf{2D Formulation}: Let us take Eq.~\ref{eq:general}. At a point of interest $\pmb{x} = (x_1, x_2)$, the gradients, Laplacian and Hessian can be represented as follows:
\begin{equation}
\mathbf{x} =
\begin{bmatrix}
x \\
y
\end{bmatrix},\quad
\mathbf{b} =
\begin{bmatrix}
b_1 \\
b_2
\end{bmatrix},\quad
\mathbf{C} =
\begin{bmatrix}
c_{11} & c_{12} \\
c_{12} & c_{22}
\end{bmatrix},
\end{equation}

\begin{equation}
\nabla \phi = \pmb{b} + 2 \, \pmb{C}\pmb{x},
\end{equation}

\begin{equation}
\nabla^2 \phi = 2\,\pmb{C}.
\end{equation}
We can generate a response matrix using perturbations in the test point along principle directions, i.e., $(x_1\pm\varepsilon, x_2\pm \varepsilon)$, along with correction-invoked and uncorrected predictions at point $\pmb{x}$. A qualitative result for the 2D Camel function (Eq.~\ref{eq:camel}) is given in Fig.~\ref{afig:camel_2d}.

\subsection{Numerical experiment details}\label{asec:num_exp}

\textbf{Training splits and Error:} The train--validation split is set to $90{:}10$ for derivative evaluations. 
The input field is normalized to $[-1,1]$ using $\max|\phi(x)|$, and the spatial grid is normalized to $[0,1]$ unless stated otherwise, in order to maintain a consistent reference frame. 
The RMSE is normalized accordingly and defined as
\begin{equation}
\mathrm{RMSE}
=
\frac{1}{\max|\phi(x)|}
\Bigl\langle\bigl(\hat{\phi}(x)-\phi(x)\bigr)^2\Bigr\rangle^{1/2},
\end{equation}
where $\hat{\phi}(x)$ and $\phi(x)$ denote the predicted and exact field values, respectively.
Additional error metrics used for the analysis of 1D Euler equation are tabulated in Tab.~\ref{tab:shock_metrics_definitions}.

\begin{table}[h]
\caption{Metrics used for shock-capturing evaluation in the Sod shock tube problem.}
\centering
\footnotesize
\renewcommand{\arraystretch}{1.3}
\begin{tabular}{lp{8cm}}
\toprule
Metric                           & Description \\
\midrule
L1 error                         & Mean absolute deviation from exact solution  \\
L$\infty$ error             & Maximum absolute deviation \\
Shock thickness                    & 10\%–90\% transition width of the discontinuity (measures smearing) \\
Post-shock oscillation            & Maximum deviation in the post-shock region  \\
Total variation                        & Sum of absolute differences in the shock window  \\
\bottomrule
\end{tabular}
\label{tab:shock_metrics_definitions}
\end{table}

\textbf{Cubic smoothing spline:} The smoothing factor $w$ for the cubic spline is selected according to the heuristic $w \approx N \sigma^2$, with $\sigma = s/3$. This approximates the noise standard deviation induced by the multiplicative Gaussian perturbation with scale $s$, where $s = \% \,\text{noise}/100$. This keeps the corrupted function value at $\phi(x) \pm s\phi(x)$.

\textbf{KBR integration into PDE solvers:} For the PDE section of the results, the inviscid Burgers' equation and the Euler equation are solved in 1D. Tables~\ref{tab:pde_ic_bc_summary} and~\ref{tab:pinn_architecture} contain the details about the problem setups and the corresponding PINN hyperparameters, respectively.

\begin{figure}[h]
    \centering
    \includegraphics[width=0.45\linewidth]{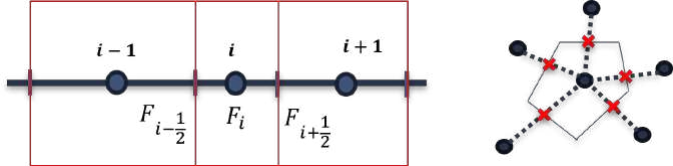}
    \caption{Schematic of (left) a one-dimensional non-uniform grid with interface flux locations, and (right) a 2D Voronoi diagram representation constructed from scattered points, representing an unstructured control volume discretization.}
    \label{afig:grids}
\end{figure}

For the inviscid Burger's equation, we use the conservative MacCormack scheme~(\cite{laney1998computational}). 
\[
\frac{\partial u}{\partial t} + \frac{\partial F(u)}{\partial x}=0,
\]
Where, $F(u) = \frac{1}{2}u^2$, is the flux. The MacCormack predictor-corrector scheme for a non-uniform grid (see Fig.~\ref{afig:grids}) is given by,
\begin{itemize}
    \item Predictor:
    \begin{equation}
U_i^{*} = U_i^n - \frac{\Delta t}{x_{i+1}-x_i}\left(F_{i+1}^n - F_i^n\right),
\end{equation}
    \item Corrector:
\begin{equation}
U_i^{n+1} = \frac{1}{2} \left[U_i^n + U_i^{*} - \frac{\Delta t}{x_i - x_{i-1}}
\left(
F_i^{*}-F_{i-1}^{*}
\right)
\right].
\end{equation}
\end{itemize}

The KBR implementation uses a finite volume-style strategy, using predicted flux value at cell interfaces. The equations then become:
\begin{itemize}
    \item Predictor:
    \begin{equation}
U_i^{*}=U_i^n-\frac{\Delta t}{x_{i+\frac{1}{2}}-x_i}\left(\hat{F}_{i+\frac{1}{2}}^n - F_i^n\right),
\end{equation}
    \item Corrector:
\begin{equation}
U_i^{n+1}= \frac{1}{2}\left[ U_i^n+U_i^{*}-\frac{\Delta t}{x_{i}-x_{i-\frac{1}{2}}}\left(F_i^{*}-\hat{F}_{i-\frac{1}{2}}^{*}\right)\right],
\end{equation}
\end{itemize}
Where, $\hat{F}_{i+\frac{1}{2}}$ and $\hat{F}_{i-\frac{1}{2}}^{*}$ are predicted fluxes at the interfaces. It can be easily simplified for the uniform grid here. In high dimensional problems with unstructured grids (see Fig.~\ref{afig:grids}), integrating the KBR learning scheme in this manner would allow to keep the schemes strictly conservative in the domain.  

The 1D Euler equation is solved using Roe's first order scheme (\cite{laney1998computational}). In a traditional sense, it uses an average central flux along with a dissipative component. 
At a point $x = x_{i-\frac{1}{2}}$, the average central flux is 
$\frac{1}{2}\left( F(x_{i-1}) + F(x_{i}) \right)$. 
The KBR-integrated scheme however does not perform this mean averaging and 
uses the predicted central flux $\hat{F}_{x-\frac{1}{2}}$ at the interface instead.

\begin{table}[h]
\centering
\renewcommand{\arraystretch}{1.1}
\small
\caption{Conservative form, initial conditions and boundary conditions for the numerical solution of the tested PDEs.}
\begin{tabular}{lccc}
\hline
\textbf{PDE} & \textbf{Conservative form} & \textbf{IC ($t=0$)} & \textbf{BC} \\
\hline
Burgers' &
$\dfrac{\partial u}{\partial t} + \dfrac{\partial}{\partial x}\left(\dfrac{u^2}{2}\right) = 0$
&
$u = \begin{cases}
1 & x<0.5 \\
0 & x\ge0.5
\end{cases}$
& $u(0,t)=1$ \\
& & & $u(1,t)= 0$ \\
\hline
Euler &
$\dfrac{\partial}{\partial t}
\begin{pmatrix} \rho \\ \rho u \\ \rho E \end{pmatrix}
+
\dfrac{\partial}{\partial x}
\begin{pmatrix} \rho u \\ \rho u^2 + p \\ u(\rho E + p) \end{pmatrix}
= \mathbf{0}$
&
$\begin{pmatrix} \rho \\ u \\ p \end{pmatrix} = \begin{cases}
\begin{pmatrix} 1 \\  0\\ 1 \end{pmatrix}& x<0.5 \\
\begin{pmatrix} 0.125\\ 0 \\ 0.1 \end{pmatrix} & x\ge0.5
\end{cases}$
& outflow \\
\hline
\end{tabular}
\label{tab:pde_ic_bc_summary}
\end{table}
\begin{table}[h]
\centering
\renewcommand{\arraystretch}{1.25}
\small
\caption{PINN architecture and training hyperparameters used for both Euler equation (Sod shock tube test) and inviscid Burgers'equation with Riemann type initial conditions.}
\begin{tabular}{lcc}
\hline
\textbf{Component} & \textbf{Value / Description} & \textbf{Notes} \\
\hline
Input dimensions & 2 & $(x, t)$ \\
Output dimensions & 1  & $u$ \\
Hidden layers & 4 & \\
Neurons per hidden layer & 64 & \\
Activation function (hidden) & $\tanh$ & \\
Output activation & softplus + $10^{-6}$ & Ensures positivity \\
Kernel initializer & Glorot Uniform & \\
Optimizer & Adam & \\
Learning rate & $10^{-3}$ & \\
Total collocation points & 60,000 & \\
Initial condition points & 600 & \\
Boundary points per side & 300 & \\
Training epochs & 50,000 & \\
\hline
\end{tabular}
\label{tab:pinn_architecture}
\end{table}

\end{document}